\documentclass[a4paper,oneside]{amsart}

\usepackage{amssymb}
\usepackage{amsmath}
\usepackage{amsfonts}
\usepackage{amsthm}
\usepackage{mathrsfs}
\usepackage{fix-cm}
\usepackage{graphicx}
\usepackage{amscd}
\usepackage{enumerate}
\usepackage{hyperref}

\usepackage{soul}
\sodef\spred{}{.2em}{.9em plus.4em}{1em plus.1em minus.1em}

\usepackage[latin1]{inputenc}
\usepackage[T1]{fontenc}

\usepackage[english]{babel}

\usepackage{bbm}

\usepackage[all]{xy}

\newbox\mybox
\def\overtag#1#2#3{\setbox\mybox\hbox{$#1$}\hbox to
  0pt{\vbox to 0pt{\vglue-#3\vglue-\ht\mybox\hbox to \wd\mybox
      {\hss$\ss#2$\hss}\vss}\hss}\box\mybox}
\def\undertag#1#2#3{\setbox\mybox\hbox{$#1$}\hbox to 0pt{\vbox to
    0pt{\vglue#3\vglue\ht\mybox\hbox to \wd\mybox
      {\hss$\ss#2$\hss}\vss}\hss}\box\mybox}
\def\lefttag#1#2#3{\hbox to 0pt{\vbox to 4pt{\vss\hbox to
      0pt{\hss$\ss#2$\hskip#3}\vss}}#1}
\def\blefttag#1#2#3{\hbox to 0pt{\vbox to 4pt{\vss\hbox to
      0pt{\hss$#2$\hskip#3}\vss}}#1}
\def\righttag#1#2#3{\hbox to 0pt{\vbox to 4pt{\vss\hbox to
      0pt{\hskip#3$\ss#2$\hss}\vss}}#1}
\let\ss\scriptstyle

\def\Dot{\lower.2pt\hbox to 3.5pt{\hss$\bullet$\hss}}
\def\Circ{\lower.2pt\hbox to 3.5pt{\hss$\circ$\hss}}

\def\splicediag#1#2{\xymatrix@R=#1pt@C=#2pt@M=0pt@W=0pt@H=0pt}

\newcommand\lineto{\ar@{-}}
\newcommand\dashto{\ar@{--}}
\newcommand\dotto{\ar@{.}}

\newtheorem{thm}{Theorem}[section]
\newtheorem{lemma}[thm]{Lemma}
\newtheorem{prop}[thm]{Proposition}

\theoremstyle{definition} 
\newtheorem{defn}[thm]{Definition} 
\newtheorem{ex}[thm]{Example}

\newcommand{\C}{\mathbbm{C}}
\newcommand{\R}{\mathbbm{R}}

\newcommand{\norm}[1]{\lVert #1 \rVert}
\newcommand{\num}[1]{\lvert #1 \rvert}
\newcommand{\morf}[4][\to]{ #2 \colon #3 #1 #4}
\newcommand{\inv}{^{-1}}

\newcommand{\M}[1]{M^{ #1 }_{m,n}}
\newcommand{\rank}{\operatorname{rank}}

\newcommand{\codim}{\operatorname{codim}}
\newcommand{\Gl}{\operatorname{GL}}

\renewcommand{\phi}{\varphi}
\renewcommand{\epsilon}{\varepsilon}

\begin{document}

\bibliographystyle{alpha}

\title{Lipschitz Normal Embeddings and Determinantal Singularities}
\author{Helge M\o{}ller Pedersen}
\address{ICMC-USP}
\author{Maria A. S. Ruas}
\address{ICMC-USP}
\email{helge@imf.au.dk}
\email{maasruas@icmc.usp.br}
\begin{abstract} 
The germ of an algebraic variety is naturally equipped with two
different metrics up to bilipschitz equivalence. The inner metric and
the outer metric. One calls a germ of a variety Lipschitz normally
embedded if the two metrics are bilipschitz equivalent. In this
article we prove that the model determinantal singularity, that is the
space of $m\times n$ matrices of rank less than a given number, is
Lipschitz normally embedded. We will also discuss some of the
difficulties extending this result to the case of general
determinantal singularities. 
\end{abstract}
\maketitle

\section{Introduction}

If $(X,0)$ is the germ of an algebraic (analytic) variety, then one
can define two natural metrics on it. Both are defined by choosing an
embedding of $(X,0)$ into $(\C^N,0)$. The first is the \emph{outer
  metric}, where the distance between two points $x,y\in X$ is given by
$d_{out}(x,y) := \norm{x-y}_{\C^N}$, so just the restriction of the
Euclidean metric to $(X,0)$. The other is the \emph{inner metric},
where the distance is defined as $d_{in}(x,y) := \inf_{\gamma}
\big{\{} length_{\C^N}(\gamma)\ \big{\vert}\ \morf{\gamma}{[0,1]}{X} \text{
  is a rectifiable curve, } \gamma(0) = x \text{ and }
\gamma(1) = y \big{\}}$. Both of these metrics are independent of the
choice of the embedding up to bilipschitz equivalence. The outer metric
determines the inner metric, and it is clear that $d_{out}(x,y) \leq
d_{in} (x,y)$. The other direction is in general not true, and we say
that $(X,0)$ is \emph{Lipschitz normally embedded} if the inner and
outer metric are bilipschitz equivalent. \emph{Bilipschitz geometry}
is the study of the bilipschitz equivalence classes of these two
metrics. Now one can of course define
the inner and outer metric for any embedded subspace of Euclidean space, but the
bilipschitz class might in this case depend of the embeddings.

In January 2016 Asuf Shachar asked the following question on
Mathoverflow.org: Is the Lie group $\Gl_n^+(\R)$ Lipschitz normally
embedded, where $\Gl_n^+(\R)$ is the group of $n\times n$ matrices
with positive determinants. A positive answer  was given by Katz,
Katz, Kerner, Liokumovich and Solomon
in \cite{kerneretc}. They first prove it for the
\emph{model determinantal singularity} $M^n_{n,n}$ (they call it the
determinantal
singularity), that is the set of $n\times n$ matrices with determinant
equal to zero. Then they replace the segments of the straight line
between two points of $\Gl_n^+(\R)$ that passes trough $\Gl_n^-(\R)$ with a line
arbitrarily close to $M^n_{n,n}$. Their proof relies on topological
arguments, and some results on conical stratifications of MacPherson
and Procesi \cite{macphersonprocesi}. In this article we give an
alternative proof relying only on linear algebra and simple
trigonometry, which also works for all model determinantal
singularities. We will also discuss the case of general determinantal
singularities, by giving some examples of determinantal singularities
that are not Lipschitz normally embedded, and then discussing some
additional assumptions on a determinantal singularity that might imply
it is Lipschitz normal embedded.    

This work is in the intersection of two areas that have seen a lot of
interest lately, namely bilipschitz geometry and determinantal
singularities. The study of bilipschitz geometry of complex spaces
started with Pham and Teissier that studied the case of curves in
\cite{phamteissier}. It then lay dormant for long time until Birbrair
and Fernandes began studying the case of complex surfaces
\cite{birbrairfernandes}. Among important recent results are the
complete classification of the inner metric of surfaces by Birbrair,
Neumann and Pichon \cite{thickthin}, the proof that Zariski
equisingularity is equivalent to bilipschitz triviality in the case of
surfaces by Neumann and Pichon \cite{zariski} and the proof that
outer Lipschitz regularity implies smoothness by Birbrair,
Fernandes, L\^{e} and Sampaio
\cite{lipschitzregularity}. Determinantal singularity is also an area
that has been around for along time, that recently saw a lot of
interest. They can be seen as a generalization of ICIS, and the recent
results have mainly been in the study of invariants coming from their
deformation theory. In \cite{ebelingguseinzade} \'Ebeling and
Guse{\u\i}n-Zade defined the index of a $1$-form, and the Milnor
number have been defined in various different ways by Ruas and da
Silva Pereira \cite{cedinhamiriam}, Damon and Pike \cite{damonpike}
and Nu\~no-Ballesteros, Or\'efice-Okamoto and Tomazalla
\cite{NunoOreficeOkamotoTomazella}. Their
deformation theory have also been studied by Gaffney and Rangachev
\cite{gaffenyrangachev} and Fr\"uhbis-Kr\"uger and Zach
\cite{fruhbiskrugerzach}.

This article is organized as follows. In section \ref{preliminaries}
we discuss the basic notions of Lipschitz normal embeddings and
determinantal singularities and give some results concerning when a
space is Lipschitz normally embedded. In section \ref{modelcase} we
prove the main theorem, that model determinantal singularities are
Lipschitz normally embedded. Finally in section \ref{secgeneralcase} we
discuss some of the difficulties to extend this result to the settings
of general determinantal singularities.

\section{Preliminaries on bilipschitz geometry and determinantal
  singularities}\label{preliminaries}

\subsection*{Lipschitz normal embeddings}

In this section we discuss some properties of Lipschitz normal
embeddings. We will first give the definition of Lipschitz normally
embedding we will work with.
\begin{defn}
We say that $X$ is \emph{Lipschitz normally embedded} if there exist
$K>1$ such that for all $x,y\in X$,
\begin{align}
d_{in}(x,y)\leq Kd_{out}(x,y).\label{lneeq}
\end{align}
We call a $K$ that satisfies the inequality \emph{a bilipschitz
  constant of} $X$.
\end{defn}

A trivial example of a Lipschitz normally embedded set is $\C^n$. For
an example of a space that is not Lipschitz normally embedded,
consider the plane curve given by $x^3-y^2=0$, then
$d_{out}((t^2,t^3),(t^2,-t^3))=2\num{t}^{\tfrac{3}{2}}$ but the
$d_{in}((t^2,t^3),(t^2,-t^3))= 2\num{t}+ o(t)$, this implies that
$\tfrac{d_{in}((t^2,t^3),(t^2,-t^3))}{d_{out}((t^2,t^3),(t^2,-t^3))}$
is unbounded as $t$ goes to $0$, hence there cannot exist a $K$
satisfying \eqref{lneeq}.

Pham and Teissier \cite{phamteissier} show that in general the outer
geometry of a complex plane curve is equivalent to its embedded
topological type, and the inner geometry is equivalent to the abstract
topological type. Hence a plane curve is Lipschitz normally embedded
if and only if it is a union of smooth curves intersecting
transversely. See also Fernandes \cite{fernandesplanecurve}.

In the cases of higher dimension the question of which singularities
are Lipschitz normally embedded becomes much more complicated. It is no
longer only rather trivial singularities that are Lipschitz normally
embedded, for example in the case of surfaces the first author
together with Neumann and Pichon, shows that rational surface
singularities are Lipschitz normally embedded if and only if they are
minimal \cite{normallyembedded}. As we will later see, determinantal
singularities give examples of non trivial Lipschitz normally
embedded singularities in arbitrary dimensions.

We will next give a couple of results about when spaces constructed from
Lipschitz normally embedded spaces are themselves Lipschitz normally
embedded. First is the case of product spaces.

\begin{prop}\label{product}
Let $X\subset \R^n$ and $Y\subset \R^m$ and let $Z = X\times Y
\subset \R^{n+m}$. $Z$ is Lipschitz normally embedded if and only if
$X$ and $Y$ are Lipschitz normally embedded.
\end{prop}

\begin{proof}
First we prove the ``if'' direction.
Let $(x_1,y_1),(x_2,y_2)\in X\times Y$. We need to show that
$d_{in}^{X\times Y}((x_1,y_1)(x_2,y_2))\leq K d_{out}^{X\times
  Y}((x_1,y_1)(x_2,y_2))$. Let $K_X$ be the constant such that
$d_{in}^X(a,b) \leq K_X d_{out}^X(a,b)$ for all $a,b\in X$, and let $K_Y$
be the constant such that $d_{in}(a,b)^Y \leq K_Y d_{out}(a,b)^Y$ for all
$a,b\in Y$. We get, using the triangle inequality, that
\begin{align*}
d_{in}^{X\times Y}((x_1,y_1)(x_2,y_2))\leq d_{in}^{X\times
  Y}((x_1,y_1)(x_1,y_2))+ d_{in}^{X\times Y}((x_1,y_2)(x_2,y_2)). 
\end{align*}
Now the points $(x_1,y_1)$ and $(x_1,y_2)$ both lie in the slice
$\{x_1\}\times Y$ and hence $d_{in}^{X\times
  Y}((x_1,y_1)(x_1,y_2)) \leq d_{in}^{Y}(y_1,y_2)$ and likewise we have
  $d_{in}^{X\times 
  Y}((x_1,y_2)(x_2,y_2)) \leq d_{in}^{X}(x_1,x_2)$. This then implies that
\begin{align*}
d_{in}^{X\times Y}((x_1,y_1)(x_2,y_2))\leq K_Y d_{out}^{Y}(y_1,y_2)+
K_X d_{out}^{X}(x_1,x_2),
\end{align*}
where we use that $X$ and $Y$ are Lipschitz normally embedded. Now it
is clear that $d_{out}^{X\times Y}((x_1,y_1)(x_1,y_2)) =
d_{out}^{Y}(y_1,y_2)$ and $d_{out}^{X\times Y}((x_1,y_2)(x_2,y_2)) =
d_{out}^{X}(x_1,x_2)$. Also, since $d_{out}^{X\times
  Y}((x_1,y_1)(x_2,y_2))^2=d_{out}^{Y}(y_1,y_2)^2+
d_{out}^{X}(x_1,x_2)^2$ by definition of the product metric, we have
that $d_{out}^{X\times Y}((x_1,y_1)(x_1,y_2)) \leq d_{out}^{X\times
  Y}((x_1,y_1)(x_2,y_2))$ and  $d_{out}^{X\times
  Y}((x_1,y_2)(x_2,y_2)) \leq d_{out}^{X\times
  Y}((x_1,y_1)(x_2,y_2))$. It then follows that
\begin{align*}
d_{in}^{X\times Y}((x_1,y_1)(x_2,y_2))\leq (K_Y + K_X) d_{out}^{X\times
  Y}((x_1,y_1)(x_2,y_2)).
\end{align*} 
For the other direction, let $p,q \in X$ consider any path
$\morf{\gamma}{ [0,1] }{Z}$ such that $\gamma(0) = (p,0)$ and
$\gamma(1) = (q,0)$. Now $\gamma(t)= \big(\gamma_X(t),\gamma_Y(t)\big)$ where
$\morf{\gamma_X}{ [0,1] }{X}$ and $\morf{\gamma_Y}{ [0,1] }{Y}$ are
paths and $\gamma_X(0) = p$ and $\gamma_X(1) =q$. Now $l(\gamma)\geq
l(\gamma_X)$, hence $d_{in}^X(p,q) \leq d_{in}^Z((p,0),(q,0))$. Now
$Z$ is Lipschitz normally embedded, so there exist a $K>1$ such that
$d_{in}^Z(z_1,z_2)\leq K d_{out}(z_1,z_2)$ for all $z_1.z_2\in Z$. We also have that
$d_{out}^Z((p,0),(q,0))= d_{out}^X (p,q)$, since $X$ is embedded in $Z$ as
  $X\times \{ 0\}$. Hence $d_{in}^X(p,q) \leq K d_{out}^X (p,q)$. The
  argument for $Y$ being Lipschitz normally embedded is the same
  exchanging $X$ with $Y$.
\end{proof}

An other case we will need later is the case of cones.

\begin{prop}\label{cone}
Let $X$ be the cone over $M$, then
$X$ is Lipschitz normally embedded if and only if $M$ is Lipschitz
normally embedded. 
\end{prop}
\begin{proof}
We first prove that $M$ Lipschitz normally embedded implies that $X$
is Lipschitz normally embedded.

Let $x,y\in X$ and assume that $\norm{x}\geq\norm{y}$. First if $x=0$
(the cone point), then the straight line from $y$ to $x$ lies in $X$,
hence $d_{in}(x,y)=d_{out}(x,y)$. So we can assume that $x\neq 0$, and
let $y'=\tfrac{y}{\norm{y}}\norm{x}$. Then $y'$ and $x$ lie in the
same copy $M_\epsilon$
of $M$, hence $d_{in}(x,y')\leq K_M d_{out}(x,y')$. Now $y'$ is the point
closest to $y$ on $M_\epsilon$. Hence all of $M_\epsilon-y'$ lies on
the other side of the affine hyperspace through $y'$ orthogonal to the
line $\overline{yy'}$ from $y$ to $y'$. Hence the angle between
$\overline{yy'}$ and the line $\overline{y'x}$ between $y'$ and $x$ is
more than $\tfrac{\pi}{2}$. Therefore, the Euclidean distance from $y$ to
$x$ is larger than $l(\overline{yy'})$ and $l(\overline{y'x})$. This gives
us:
\begin{align*}
d_{in}(x,y) &\leq d_{in}(x,y')+ d_{in}(y',y) \leq
K_md_{out}(x,y')+d_{out}(y',y)\\ &\leq (K_m+1)d_{out}(x,y).
\end{align*}

For the other direction, assume that $X$ is Lipschitz normally
embedded, but $M$ is not Lipschitz normally embedded. 

Since $M$ is compact the only obstructions to being Lipschitz normally
embedded are local. So let $p\in M$ be a point such that $M$ is not
Lipschitz normally embedded in a small open neighbourhood $U\subset M$ of $p$. By
Proposition \ref{product} we have that $U\times (-\epsilon,\epsilon)$ is not
Lipschitz normally embedded, where $0<\epsilon$ is much smaller than
the distance from $M$ to the origin. Now the quotient map from
$\morf{c}{M\times [0,\infty)}{X}$ induces an outer (and therefore also
inner) bilipschitz equivalence of $U\times (-\epsilon,\epsilon)$ with
$c\big(U\times (-\epsilon,\epsilon)\big)$. Since both $U$ and
$\epsilon$ can be chosen to be arbitrarily small, we have that there
does not exist any small open neighbourhood of $p\in X$ that is
Lipschitz normally embedded, contradicting that $X$ is Lipschitz
normally embedded. Hence $X$ being Lipschitz normally embedded implies
that $M$ is Lipschitz normally embedded.  

\end{proof}

\subsection*{Determinantal singularities}

Let $M_{m,n}$ be the space of $m\times n$ matrices with complex (or
real) entries. For $1\leq t\leq \min \{ m,n \}$ let $\M{t}$ denote the
\emph{model determinantal
  singularity}, that is $\M{t} = \big{\{} A\in M_{m,n} \vert \rank A <
  t \big{\}}$. $\M{t}$ is an algebraic variety, with algebraic
  structure defined by the vanishing of all $t\times t$ minors. It is
  homogeneous, and hence a real cone over its real link, it is also
  a complex cone but it is the real conical structure we will use.  
It is highly singular with the singular set of $\M{t}$ being
$\M{t-1}$. If fact the action of the group $\Gl_m\times\Gl_n$ by
conjugation insures that the decomposition $\M{t} = \bigcup_{i=1}^t
\M{i}-\M{i-1}$ is a Whitney stratification.   

Let $\morf{F}{\C^N}{M_{m,n}}$ be a map with holomorphic entries, then
  $X=F\inv (\M{t})$ is a \emph{determinantal variety} of type
  $(m,n,t)$ if $\codim X = \codim
  \M{t} = (m-t+1)(n-t+1)$. If $F(0)= 0$ we will call the germ $(X,0)$
  a \emph{determinantal singularity} of type $(m,n,t)$.

Determinantal singularities can have quite bad singularities, hence one
often restrict to the following subset with better properties:
\begin{defn}
Let $X$ be a determinantal singularity defined by a map
$\morf{F}{\C^N}{M_{m,n}}$. One says that $X$ is an \emph{essentially
  isolated determinantal singularity} (EIDS for short) if $F$ is
transversal to the strata of $\M{t}$ at all point in a punctured
neighbourhood of the origin.
\end{defn}
While an EIDS can still have very bad singularities at the origin, it
singular points away from the origin only depends on the type and $N$,
for example $X-F\inv(\M{1})$ has a nice action, and is stratified by
$X^i= F\inv(\M{i}- \M{i-1})$. A lot of interesting singularities are
EIDS, for example all ICIS are EIDS.

In proving that our singularities are Lipschitz normally embedded,
we often have to change coordinates, to get some nice matrices for our
points. Hence we need the following lemma to see that these changes of
coordinates do preserve the inequalities we are using.

\begin{lemma}\label{chageofcoordinates}
Let $V\subset M_{m,n}$ be a subset invariant under linear change of
coordinates. If $x,y\in V$ satisfy $d_{in}(x,y)\leq Kd_{out}(x,y)$,
then the same is true after any linear change of coordinates.
\end{lemma}
\begin{proof}
Any linear change of coordinates of $M_{m,n}$ is given by conjugation
by a pair of matrices $(A,B)\in\Gl_m\times\Gl_n$. 

First we see that the outer metric is just scale by the following:
\begin{align*}
d_{out}(AxB\inv,AyB\inv) &=\norm{A(x-y)B\inv}
=\norm{A}\norm{x-y}\norm{B\inv}\\ &=\norm{A}\norm{B\inv}d_{out}(x,y). 
\end{align*}

Second we consider the case of length of curves. Let $P^V_{x,y}:=\{
\morf{\gamma}{[0,1]}{V}\ \vert\ \gamma\text{ is a
  rectifiable curve such that } \gamma(0)=x
\text{ and } \gamma(1)=y\}$, then the conjugation $\gamma\to A\gamma B\inv$
defines a bijection of $P^V_{x,y}$ and $P^{AVB\inv}_{AxB\inv,AyB\inv} =
P^V_{AxB\inv,AyB\inv}$. Moreover, that 
$l(A\gamma B\inv)=\norm{A}l(\gamma)\norm{B\inv}$ follows from the
definition of length of a curve. Hence
\begin{align*}
d_{in}(x,y) &= \inf_{\gamma\in P_{x,y}} \big{\{} l(\gamma) \big{\}} = 
\inf_{A\gamma B\inv \in P_{AxB\inv,AyB\inv}} \bigg{\{} \frac{l(A\gamma
  B\inv)}{\norm{A}\norm{B\inv}} \bigg{\}}\\ &=
  \frac{d_{in}(AxB\inv,AyB\inv)}{\norm{A}\norm{B\inv}}.
\end{align*}
The result then follows.
\end{proof}

\section{The case of the model determinantal singularities}\label{modelcase}

In this section we prove that $\M{t}$ is Lipschitz normally
embedded. We do that by considering several cases for the position of
two points $p,q\in \M{t}$, and finding inequalities of the form
$d_{in}(p,q) \leq K d_{out}(p,q)$, where we explicitly give the value
of $K$. First we consider the simple case
where $q=0$. 

\begin{lemma}\label{0case}
Let $p\in\M{t}$ then $d_{in}(p,0)=d_{out}(p,0)$.
\end{lemma}

\begin{proof}
This follows since $\M{t}$ is conical, and hence the straight line
from $p$ to $0$ lies in $\M{t}$.
\end{proof}

The second case we consider is when $p$ and $q$ are orthogonal. This
case is not much more complicated than the case $q=0$.

\begin{lemma}\label{orthogonalcase}
Let $p,q\in\M{t}$ such that $\langle p,q\rangle = 0$. Then
$d_{in}(p,q) \leq 2d_{out}(p,q)$.
\end{lemma}

\begin{proof}
That $\langle p,q \rangle = 0$ implies
that the line from $p$ to $q$, the line from $p$ to $0$ and the line
from $q$ to $0$ form a right triangle with the line from $p$ to
$q$ as the hypotenuse. Hence $d_{out}(p,0)\leq d_{out}(p,q)$ and
$d_{out}(q,0)\leq d_{out}(p,q)$, this the gives that:
\begin{align*}
d_{in}(p,q)\leq d_{in}(p,0)+ d_{in}(q,0) = d_{out}(p,0)+ d_{out}(q,0)
\leq 2 d_{out}(p,q).
\end{align*}
 
\end{proof}

The last case we need to consider is the case where $p$ and $q$ are
not orthogonal. This case is a little more complicated and we need to
do the proof by induction. 

\begin{lemma}\label{generalcase}
Let $p,q\in\M{t}$ such that $\langle p,q \rangle \neq 0$. Then
$d_{in}(p,q)\leq 2\rank(p)d_{out}(p,q)$.
\end{lemma}
\begin{proof}
The poof is by induction in $t$ by considering $\M{t}$ as
depending on $t$ as $M^{t}_{m'+t,n'+t}$. The base case is $\M{1}=\{
0\}$, which trivially satisfies the inequality. 
So we assume the theorem is true for $M^{t-1}_{m-1,n-1}$.

By a change of coordinates we can assume that $p$ and $q$ have the
following forms:
\begin{align*}
p=\left(
\begin{array}{@{} c | c @{}}
p_{11} & \begin{matrix} p_{12} & \dots & p_{1n} \end{matrix} \\ \hline
\begin{matrix} 0 \\ \vdots \\ 0 \end{matrix} & \text{\Huge $D_p$}
\end{array}
\right)\text{ and } q=\left(
\begin{array}{@{} c | c @{}}
q_{11} & \begin{matrix} 0 & \dots & 0 \end{matrix} \\ \hline
\begin{matrix} q_{21} \\ \vdots \\ q_{m1} \end{matrix} & \text{\Huge $D_q$}
\end{array}
\right)
\end{align*}
where $p_{11},q_{11}\neq 0$ and $D_p,D_q\in M_{m-1,n-1}$. Then let
$p'$, $q'$ and $q_0$ be the following points:
\begin{align*}
p'=\left(
\begin{array}{@{} c | c @{}}
q_{11} & \begin{matrix} 0 & \dots & 0 \end{matrix} \\ \hline
\begin{matrix} 0 \\ \vdots \\ 0 \end{matrix} & \text{\Huge $D_p$}
\end{array}
\right),\ q'=\left(
\begin{array}{@{} c | c @{}}
q_{11} & \begin{matrix} 0 & \dots & 0 \end{matrix} \\ \hline
\begin{matrix} 0 \\ \vdots \\ 0 \end{matrix} & \text{\Huge $D_q$}
\end{array}
\right)\text{ and } 
q_0=\left(
\begin{array}{@{} c | c @{}}
q_{11} & \begin{matrix} 0 & \dots & 0 \end{matrix} \\ \hline
\begin{matrix} 0 \\ \vdots \\ 0 \end{matrix} & \text{\Huge $0$}
\end{array}
\right).
\end{align*}
It is clear that $p',q'\in \M{t}$, moreover the straight line
$\overline{pp'}$ from $p$ to $p'$ is in $\M{t}$, and the straight line
$\overline{qq'}$ from $q$ to $q'$ is in $\M{t}$. Hence
$d_{in}(p,p')=d_{out}(p,p')$ and $d_{in}(q,q')=d_{out}(q,q')$. Let
$H_{q_0}$ be the affine space through $q_0$ defined as
\begin{align*}
H_{q_0} := \Bigg{\{} \left(
\begin{array}{@{} c | c @{}}
q_{11} & \begin{matrix} 0 & \dots & 0 \end{matrix} \\ \hline
\begin{matrix} 0 \\ \vdots \\ 0 \end{matrix} & \text{\Huge $A$}
\end{array}
\right)\in M_{m,n} \Big\vert \text{ where } A\in M_{m-1,n-1} \Bigg{\}}.
\end{align*}
It is clear that $p',q'\in H_{q_0}$ and hence $p',q'\in H_{q_0}\bigcap
\M{t}$. Now $H_{q_0}\bigcap \M{t}$ is isomorphic to
$M^{t-1}_{m-1,n-1}$, and we get by induction  $d_{in}(p',q') \leq
2\rank D_p d_{out}(p',q') = 2(\rank p-1) d_{out}(p',q') $. 

We now have that:
\begin{align}\label{inequalitygeneric}
d_{in}(p,q) &\leq d_{in}(p,q')+ d_{in}(q',q) \leq d_{in}(p,p') +
d_{in}(p',q')+ d_{in}(q',q) \\ &\leq d_{out}(p,p') + 2(rank
p-1)d_{out}(p',q') + d_{out}(q',q).\nonumber
\end{align}
The line $\overline{pp'}$ is in the direction $p-p'$ and the line
$\overline{p'q '}$ is in the direction $q'-p'$. These direction are
\begin{align*}
p-p'=\left(
\begin{array}{@{} c | c @{}}
p_{11}-q_{11} & \begin{matrix} p_{12} & \dots & p_{1n} \end{matrix} \\ \hline
\begin{matrix} 0 \\ \vdots \\ 0 \end{matrix} & \text{\Huge $0$}
\end{array}
\right)\text{ and } q'-p'=\left(
\begin{array}{@{} c | c @{}}
0 & \begin{matrix} 0 & \dots & 0 \end{matrix} \\ \hline
\begin{matrix} 0 \\ \vdots \\ 0 \end{matrix} & \text{\Huge
  $D_q-D_p$}
\end{array}
\right),
\end{align*}
hence $\overline{pp'}$ and $\overline{p'q '}$ are orthogonal. This
implies that the straight line $\overline{pq'}$ is the hypotenuse of a
right triangle given by $p,p'$ and $q'$. We therefore have that
$d_{out}(p,p') \leq d_{out}(p,q')$ and $d_{out}(p',q') \leq
d_{out}(p,q')$. 

Likewise we have that the line $\overline{pq'}$ is in the direction
$p-q'$ and the line $\overline{qq '}$ is in the direction
$q-q'$. These direction are 
\begin{align*}
p-q'=\left(
\begin{array}{@{} c | c @{}}
p_{11}-q_{11} & \begin{matrix} p_{12} & \dots & p_{1n} \end{matrix} \\ \hline
\begin{matrix} 0 \\ \vdots \\ 0 \end{matrix} & \text{\Huge $D_p-D_q$}
\end{array}
\right)\text{ and } q-q'=\left(
\begin{array}{@{} c | c @{}}
0 & \begin{matrix} 0 & \dots & 0 \end{matrix} \\ \hline
\begin{matrix} q_{21} \\ \vdots \\ q_{m1} \end{matrix} & \text{\Huge
  $0$}
\end{array}
\right),
\end{align*}
so $\overline{pq'}$ and $\overline{qq '}$ are orthogonal. Hence we
have that $p,q$ and $q'$ form a right triangle with $\overline{pq}$ as
hypotenuse, which implies that $d_{out}(p,p') \leq d_{out}(p,q')$ and
$d_{out}(p',q') \leq d_{out}(p,q')$. When we combine this with the
previous 
paragraph it follows that $d_{out}(p,p'), d_{out}(q',p'), d_{out}(q,q') \leq
d_{out}(p,q)$, and then using this in inequality
\eqref{inequalitygeneric} the result follows. 

\end{proof}

We have now considered all possible pairs of $p,q$, and we can then
combine the results to get the main theorem.

\begin{thm}
The model determinantal singularity $\M{t}$ is Lipschitz normally
embedded, with a bilipschitz constant $2t-2$. 
\end{thm}
\begin{proof}
Let $p,q\in\M{t}$. If $\langle p,q \rangle=$ then $d_{in}(p,q)\leq
2d_{out}(p,q)$ by Lemma
\ref{orthogonalcase}. If $\langle p,q \rangle \neq 0$ then
$d_{in}(p,q)\leq 2\rank(p)d_{out}(p,q)$ by Lemma
\ref{generalcase}. Hence in all cases $d_{in}(p,q)\leq (2t-2)d_{out}(p,q)$
since $\rank p\leq t-1$.
\end{proof}

\section{The general case}\label{secgeneralcase}

The case of a general determinantal singularity is much more difficult
that the case of a model one. One can in general not expect a determinantal
singularity to be Lipschitz normally embedded, the easiest way to see
this is to note that all ICIS are determinantal, and that there are
many ICIS that are not Lipschitz normally embedded. For example
among the simple surface singularities $A_n$, $D_n$, $E_6$, $E_7$ and
$E_8$ only the $A_n$'s are Lipschitz normally embedded. Since the
structure of determinantal singularities does not give us any new
tools to study ICIS, we will probably not be able to say when an ICIS
is Lipschitz normally embedded. This means that since $F\inv(\M{1})$
is often an ICIS, we probably have to assume it is Lipschitz normally
embedded to say 
anything about whether $F\inv(\M{t})$ is Lipschitz normally
embedded. But before we discuss such assumption further, we will give
some examples of determinantal singularities that fails to be
Lipschitz normally embedded. 

\begin{ex}\label{degenerationofcusps}
Let $X$ be the determinantal singularity of type $(3,3,3)$ given by
the following map $\morf{F}{\C^3}{M_{3,3}}$:
\begin{align*}
F(x,y,z)=\left(
\begin{array}{@{} c c c @{}}
x & 0 & z \\ 
y & x & 0 \\ 
0 & y & x
\end{array}
\right).
\end{align*}
Since this is a linear embedding of $\C^3$ into $\C^9$, one can see
$X$ as an intersection of a linear subspace and $M^{3}_{3,3}$. Hence
one would expect it to be a nice space. On the other hand
$X=V(x^3-y^2z)$, hence it is a family of cusps degeneration to a
line, or seeing an other way as a cone over a cusp. But $X$ being
Lipschitz normally embedded would imply that the cusp $x^3-y^2=0$ 
is Lipschitz normally embedded by
Proposition \ref{cone}, which we know it is
not by the work of Pham and Teissier \cite{phamteissier}.
\end{ex}   

Notice that in the Example \ref{degenerationofcusps},
$X^1=F\inv(M^{1}_{3,3})$ is a point and $X^1=F\inv(M^{2}_{3,3})$ is
a line, so both $X^1$ and $X^2$ are Lipschitz normally embedded. So it
does not in general follows that if $X^i$ is Lipschitz normally
embedded then $X^{i+1}$ is. Now the singularity in Example
\ref{degenerationofcusps} is not an EIDS, $F$ is not transverse to the
strata of $M^3_{3,3}$ at points on the $z$-axis. In the next example
we will see that EIDS is not enough either.

\begin{ex}[Simple Cohen-Macaulay codimensional 2 surface singularities]\label{scmc2ss}
In \cite{fruhbiskrugerneumer} Fr\"uhbis-Kr\"uger and Neumer classify
simple Cohen-Macaulay codimension 2 singularities. They are all EIDS of
type $(3,2,2)$, and the surfaces correspond to the
rational triple points classified by Tjurina \cite{tjurina}. We will
look closer at two of such families. First we have the family
given by the matrices:
\begin{align*}
\left(
\begin{array}{@{} c c c @{}}
z & y+w^l & w^m \\ 
w^k & y & x
\end{array}
\right).
\end{align*} 
This family corresponds to the family of triple points in
\cite{tjurina} called $A_{k-1,l-1,m-1}$. Tjurina shows that the dual
resolution graph of their minimal resolution are:
$$
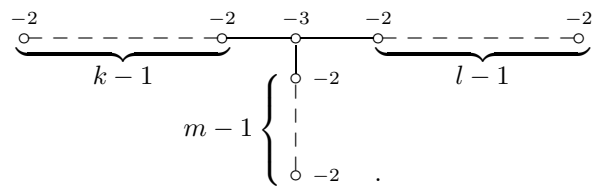
\xymatrix@R=6pt@C=24pt@M=0pt@W=0pt@H=0pt{
&&\\
\overtag{\Circ}{-2}{8pt}\dashto[rr] &
{\hbox to 0pt{\hss$\underbrace{\hbox to 80pt{}}$\hss}}&
\overtag{\Circ}{-2}{8pt}\lineto[r] &
\overtag{\Circ}{-3}{8pt}\lineto[r]\lineto[d] &
\overtag{\Circ}{-2}{8pt}\dashto[rr] &
{\hbox to 0pt{\hss$\underbrace{\hbox to 80pt{}}$\hss}}&
\overtag{\Circ}{-2}{8pt}\\
&{k-1}&& \righttag{\Circ}{-2}{8pt}\dashto[dddd] &&{l-1}\\
&&&&\\
&&&&\blefttag{\quad}{m-1\begin{cases} \quad \\
    \ \\ \ \end{cases}}{10pt} & \\
&&&&\\
&&& \righttag{\Circ}{-2}{8pt} & .\\
&&}$$ 
Using Remark 2.3 of \cite{spivakovsky} we see that these singularities
are minimal, and hence by the result of \cite{normallyembedded} we get
that they are Lipschitz normally embedded.

The second family is given by the matrices:
\begin{align*}
\left(
\begin{array}{@{} c c c @{}}
z & y+w^l & xw \\ 
w^k & x & y
\end{array}
\right).
\end{align*} 
Tjurina calls this family $B_{2l,k-1}$ and give the dual resolution
graphs of their minimal resolutions as:
$$
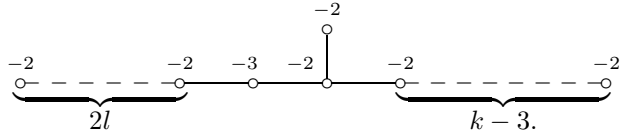
\xymatrix@R=6pt@C=24pt@M=0pt@W=0pt@H=0pt{
&&&& \overtag{\Circ}{-2}{8pt} &\\
&&&&&\\
\overtag{\Circ}{-2}{8pt}\dashto[rr] &
{\hbox to 0pt{\hss$\underbrace{\hbox to 65pt{}}$\hss}}&
\overtag{\Circ}{-2}{8pt}\lineto[r] &
\overtag{\Circ}{-3 \hspace{7pt}}{8pt}\lineto[r] &
\overtag{\Circ}{-2 \hspace{20pt}}{8pt}\lineto[r] \lineto[uu]&
\overtag{\Circ}{-2}{8pt}\dashto[rr]
&{\hbox to 0pt{\hss$\underbrace{\hbox to 80pt{}}$\hss}}&
\overtag{\Circ}{-2}{8pt}\\
&2l&&& &&k-3.& \\
&&}$$ 
Following Spivakovsky this is not a minimal singularity, and since it
is rational according to Tjurina it is not Lipschitz normally embedded
by the result of \cite{normallyembedded}.

These two families do not look very different but one is Lipschitz
normally embedded and the other is not. We can do the same for all the
Cohen-Macaulay codimension 2 surfaces, and using the results in
\cite{normallyembedded} that rational surface singularities are
Lipschitz normally embedded if and only if they are minimal, we get
that only the family $A_{l,k,m}$ is Lipschitz normally
embedded. This is similar to the case of codimension 1, since only the
$A_n$ singularities are Lipschitz normally embedded among the simple
singularities. 
\end{ex}

So as we see in Example \ref{scmc2ss} being and EIDS with singular
set Lipschitz normally embedded, is not enough to ensure the variety is
Lipschitz normally embedded. One should notice that the varieties
in Example \ref{degenerationofcusps} and \ref{scmc2ss} are both
defined by maps $\morf{F}{\C^N}{M_{m,n}}$ where $N<mn$. This
means that one should think of the singularity as a section of
$\M{t}$, but being a subspace of a Lipschitz normally embedded space
does not imply the Lipschitz normally embedded condition. 
If $N\geq mn$ then 
one can think about the singularity being a fibration over $\M{t}$,
and as we saw in Proposition \ref{product} products of Lipschitz
normally embedded spaces are Lipschitz normally embedded. Now in this
case $X^1=F\inv(\M{1})$ is ICIS if $X$ is an EIDS, which means that we
probably can not say anything general about whether it is Lipschitz
normally embedded or not. So natural assumptions would be to assume
that $X$ is an EIDS and that $X^1$ is Lipschitz normally embedded.

\section*{Acknowledgements}
The authors would like to thank Anne Pichon for first letting us know
about Asuf Shachar's question on Mathoverlfow.org, and later pointing
out a mistake in the proof of the first version of Proposition
\ref{generalcase}. We would also like to thank Lev Birbrair for
sending us an early version of the paper by Katz, Katz, Kerner,
Liokumovich and Solomon \cite{kerneretc}, and encouraging us to work on
the problem. The first author was supported by FAPESP grant
2015/08026-4 and the second author was partially supported by FAPESP
grant 2014/00304-2 and CNPq grant 306306/2015-8.

\bibliography{general}

\newcommand{\etalchar}[1]{$^{#1}$}
\def\cprime{$'$}
\begin{thebibliography}{NBOOT13}

\bibitem[BF08]{birbrairfernandes}
Lev Birbrair and Alexandre Fernandes.
\newblock Inner metric geometry of complex algebraic surfaces with isolated
  singularities.
\newblock {\em Comm. Pure Appl. Math.}, 61(11):1483--1494, 2008.

\bibitem[BFLS16]{lipschitzregularity}
Lev Birbrair, Alexandre Fernandes, D{\~u}ng~Tr{\'a}ng L{\^e}, and J.~Edson
  Sampaio.
\newblock Lipschitz regular complex algebraic sets are smooth.
\newblock {\em Proc. Amer. Math. Soc.}, 144(3):983--987, 2016.

\bibitem[BNP14]{thickthin}
Lev Birbrair, Walter~D. Neumann, and Anne Pichon.
\newblock The thick-thin decomposition and the bilipschitz classification of
  normal surface singularities.
\newblock {\em Acta Math.}, 212(2):199--256, 2014.

\bibitem[DP14]{damonpike}
James Damon and Brian Pike.
\newblock Solvable groups, free divisors and nonisolated matrix singularities
  {II}: {V}anishing topology.
\newblock {\em Geom. Topol.}, 18(2):911--962, 2014.

\bibitem[Fer03]{fernandesplanecurve}
Alexandre Fernandes.
\newblock Topological equivalence of complex curves and bi-{L}ipschitz
  homeomorphisms.
\newblock {\em Michigan Math. J.}, 51(3):593--606, 2003.

\bibitem[FKN10]{fruhbiskrugerneumer}
Anne Fr{\"u}hbis-Kr{\"u}ger and Alexander Neumer.
\newblock Simple {C}ohen-{M}acaulay codimension 2 singularities.
\newblock {\em Comm. Algebra}, 38(2):454--495, 2010.

\bibitem[FZ15]{fruhbiskrugerzach}
Anne {Fruehbis-Krueger} and Matthias {Zach}.
\newblock {On the Vanishing Topology of Isolated Cohen-Macaulay Codimension 2
  Singularities}.
\newblock {\em ArXiv e-prints}, January 2015.

\bibitem[GR15]{gaffenyrangachev}
Terence {Gaffney} and Antoni {Rangachev}.
\newblock {Pairs of modules and determinantal isolated singularities}.
\newblock {\em ArXiv e-prints}, December 2015.

\bibitem[GZ{\`E}09]{ebelingguseinzade}
Sabir~M. Guse{\u\i}n-Zade and Wolfgang {\`E}beling.
\newblock On the indices of 1-forms on determinantal singularities.
\newblock {\em Tr. Mat. Inst. Steklova}, 267(Osobennosti i
  Prilozheniya):119--131, 2009.

\bibitem[KKKLS16]{kerneretc}
Karin~U. {Katz}, Mikhail~G. {Katz}, Dmitry {Kerner}, Yevgeny {Liokumovich}, and
  Jake~P. {Solomon}.
\newblock {Determinantal variety and bilipschitz equivalence}.
\newblock {\em ArXiv e-prints}, February 2016.

\bibitem[MP98]{macphersonprocesi}
Robert MacPherson and Claudio Procesi.
\newblock Making conical compactifications wonderful.
\newblock {\em Selecta Math. (N.S.)}, 4(1):125--139, 1998.

\bibitem[NBOOT13]{NunoOreficeOkamotoTomazella}
Juan~J. Nu{\~n}o-Ballesteros, Bruna Or{\'e}fice-Okamoto, and Jo{\~a}o~N.
  Tomazella.
\newblock The vanishing {E}uler characteristic of an isolated determinantal
  singularity.
\newblock {\em Israel J. Math.}, 197(1):475--495, 2013.

\bibitem[NP14]{zariski}
Walter~D. Neumann and Anne Pichon.
\newblock Lipschitz geometry of complex surfaces: analytic invariants and
  equisingularity.
\newblock {\em arxiv:1211.4897}, 2014.

\bibitem[NPP15]{normallyembedded}
Walter~D. Neumann, Helge~M\o{}ller Pedersen, and Anne Pichon.
\newblock Minimal surface singularities are {L}ipschitz normally embedded.
\newblock {\em arxiv:1503.03301}, 2015.

\bibitem[PT69]{phamteissier}
Fr{\'e}d{\'e}ric Pham and Bernard Teissier.
\newblock {Fractions Lipschitziennes d'une alg{\`e}bre analytique complexe et
  saturation de Zariski, par Fr{\'e}d{\'e}ric Pham et Bernard Teissier}.
\newblock 42 pages. Ce travail est la base de l'expos{\'e} de Fr{\'e}d{\'e}ric
  Pham au Congr{\`e}s International des Math{\'e}maticiens, Nice 1970., June
  1969.

\bibitem[Spi90]{spivakovsky}
Mark Spivakovsky.
\newblock Sandwiched singularities and desingularization of surfaces by
  normalized {N}ash transformations.
\newblock {\em Ann. of Math. (2)}, 131(3):411--491, 1990.

\bibitem[SRDSP14]{cedinhamiriam}
Maria~Aparecida Soares~Ruas and Miriam Da~Silva~Pereira.
\newblock Codimension two determinantal varieties with isolated singularities.
\newblock {\em Math. Scand.}, 115(2):161--172, 2014.

\bibitem[Tju68]{tjurina}
Galina~N. Tjurina.
\newblock Absolute isolation of rational singularities, and triple rational
  points.
\newblock {\em Funkcional. Anal. i Prilo\v zen.}, 2(4):70--81, 1968.

\end{thebibliography}

\end{document}